\documentclass[11pt]{article}
\usepackage{amssymb,amsmath,amsfonts,enumerate,latexsym,pifont}
\usepackage{amsthm}
\usepackage{mathptmx}
\usepackage{gastex}
\usepackage{url}

\theoremstyle{plain}
\newtheorem{theorem}{Theorem}
\newtheorem{lemma}[theorem]{Lemma}

\theoremstyle{remark}
\newtheorem{remark}{Remark}

\newcommand{\ais}{ai-semi\-ring}
\def\softd{{\leavevmode\setbox1=\hbox{d}%
    \hbox to 1.05\wd1{d\kern-0.4ex{\char039}\hss}}}

\title{Semiring identities of the Brandt monoid\thanks{Supported by the Ministry of Science and Higher Education of the Russian Federation (Ural Mathematical Center project No. 075-02-2020-1537/1)}}
\author{Mikhail Volkov\\
Institute of Natural Sciences and Mathematics\\
Ural Federal University\\
m.v.volkov@urfu.ru}

\date{}

\begin{document}

\maketitle

\begin{abstract}
The 6-element Brandt monoid $B_2^1$ admits a unique addition under which it becomes an additively idempotent semiring. We show that this addition is a term operation of $B_2^1$ as an inverse semigroup. As a consequence, we exhibit an easy proof that the semiring identities of $B_2^1$ are not finitely based.
\end{abstract}

We assume the reader's acquaintance with basic concepts of universal algebra such as an identity and a variety; see, e.g., \cite[Chapter~II]{BuSa81}.

The 6-element Brandt monoid $B_2^1$ can be represented as a semigroup of the following zero-one $2\times 2$-matrices
\begin{equation}
\label{eq:b21}
\begin{tabular}{cccccc}
$\left(\begin{matrix} 0&0\\0&0\end{matrix}\right)$
&
$\left(\begin{matrix} 1&0\\0&1\end{matrix}\right)$
&
$\left(\begin{matrix} 0&1\\0&0\end{matrix}\right)$
&
$\left(\begin{matrix} 0&0\\1&0\end{matrix}\right)$
&
$\left(\begin{matrix} 1&0\\0&0\end{matrix}\right)$
&
$\left(\begin{matrix} 0&0\\0&1\end{matrix}\right)$
\\
\rule{0cm}{.5cm}$0$&$E$&$E_{12}$&$E_{21}$&$E_{11}$&$E_{22}$
\end{tabular}
\end{equation}
under the usual matrix multiplication $\cdot$ or as a monoid with presentation
\[
\langle E_{12},E_{21}\mid E_{12}E_{21}E_{12}=E_{12},\ E_{21}E_{12}E_{21}=E_{21},\ E_{12}^2=E_{21}^2=0\rangle.
\]
Quoting from a recent paper \cite{JackZhang21}, `This Brandt monoid is perhaps the most ubiquitous harbinger of complex behaviour in all finite semigroups'. In particular, $(B_2^1,\cdot)$ has no finite basis for its identities (Perkins \cite{Per66,Per69}) and is one of the four smallest semigroups with this property (Lee and Zhang \cite{LeeZhang2015}).

The monoid $(B_2^1,\cdot)$ has a natural involution that swaps $E_{12}$ and $E_{21}$ and fixes all other elements. In terms of the matrix representation \eqref{eq:b21} this involution is nothing but the usual matrix transposition; we will, however, use the notation $x\mapsto x^{-1}$ for the involution, emphasizing that $x^{-1}$ is the unique inverse of $x$. Recall that elements $x,y$ of a semigroup $(S,\cdot)$ are said to be \emph{inverses} of each other if $xyx=x$ and $yxy=y$. A~semigroup is called \emph{inverse} if every its element has a unique inverse; inverse semigroups can therefore be thought of as algebras of type (2,1). Being considered as an inverse semigroup, the monoid $(B_2^1,\cdot,{}^{-1})$ retains its complex equational behaviour: $B_2^1$ has no finite basis for its inverse semigroup identities (Kleiman \cite{Kle79}) and is the smallest inverse semigroup with this property (Kleiman \cite{Kle77,Kle79}).

In the present note we consider equational properties of yet another enhancement of the monoid $(B_2^1,\cdot)$ with an additional operation, this time binary. Recall that an \emph{additively idempotent semiring} an algebra $(S, +, \cdot)$ of type $(2, 2)$ such that the additive reduct $(S,+)$ is a \emph{semilattice} (that is, a commutative idempotent semigroup), the multiplicative reduct $(S,\cdot)$ is a semigroup, and multiplication distributes over addition on the left and on the right, that is, $(S,+,\cdot)$ satisfies the identities $x(y+z)\approx xy+xz$ and $(y+z)x\approx yx+zx$. In papers which motivation comes from semigroup theory, objects of this sort sometimes appear under the name \emph{semilattice-ordered semigroups}, see, e.g., \cite{KP05} or \cite{McA07}. We will stay with the term `additively idempotent semiring', abbreviated to `\ais' in the sequel.

Our key observation is the following:

\begin{lemma}
\label{lem:b21-as-ais}
Let $(S,\cdot,{}^{-1})$ be an inverse semigroup satisfying the identity
\begin{equation}
\label{eq:comb}
x^n\approx x^{n+1}
\end{equation}
for some $n$. Define
\[
x\oplus y:=(xy^{-1})^nx.
\]
Then $(S,\cdot,\oplus)$ is an \ais.
\end{lemma}

\begin{proof}
Let $E(S)$ stand for the set of all idempotents of $S$. The relation
\[
\le:= \{(a,b)\in S\times S\mid a=eb\ \text{ for some } e\in E(S)\}
\]
is a partial order on $S$ referred to as the \emph{natural partial order}; see \cite[Section II.1]{Pet84} or \cite[pp. 21--23]{Law99}. We need two basic properties of the natural partial order:

1) $\le$ is compatible with both multiplication and inversion;

2) $a\le b$ if and only if $a=bf$ for some $f\in E(S)$.

Take any $a,b\in S$ and suppose that $c\le a$ and $c\le b$. Then $c^{-1}\le b^{-1}$ whence by the compatibility with multiplication
\[
c=(cc^{-1})^nc\le (ab^{-1})^na=a\oplus b.
\]
In presence of the identity \eqref{eq:comb}, $(ab^{-1})^n=(ab^{-1})^{n+1}=\cdots=(ab^{-1})^{2n}$. Hence
\[
a\oplus b=(ab^{-1})^n\cdot a\le a.
\]
Further,
\begin{multline*}
a\oplus b=(ab^{-1})^n a=(ab^{-1})^{n+1} a=\cdots=(ab^{-1})^{2n-1} a={} \\
(ab^{-1})^{n}\cdot (ab^{-1})^{n-1}a = (ab^{-1})^n\cdot a(b^{-1}a)^{n-1}={}\ \text{ (using $b^{-1}=b^{-1}bb^{-1}$)}\\
(ab^{-1})^n\cdot b\cdot (b^{-1}a)^n\le b\cdot (b^{-1}a)^n\le b
\end{multline*}
since $(b^{-1}a)^n\in E(S)$. We see that $a\oplus b$ is nothing but the infimum of $\{a,b\}$ with respect to the natural partial order. Thus, $(S,\oplus)$ is a semilattice. It is known \cite[Proposition 1.22]{Sch73}, see also \cite[Proposition 19]{Law99} that if a subset $H\subseteq S$ possesses an infimum under the natural partial order, then so do the subsets $sH$ and $Hs$ for any $s\in S$, and $\inf(sH) = s(\inf H)$, $\inf(Hs) = (\inf H)s$. This implies that multiplication distributes over $\oplus$ on the left and on the right.
\end{proof}

\begin{remark}
The essence of Lemma~\ref{lem:b21-as-ais} is known. Leech, in the course of his comprehensive study of inverse monoids $(S,\cdot,{}^{-1},1)$ that are inf-semilattices under the natural partial order, has verified that $(S,\le)$ is a inf-semilattice whenever $S$ is a periodic combinatorial\footnote{A semigroup $S$ is \emph{periodic} if all monogenic subsemigroups of $S$ are finite and \emph{combinatorial} if all subgroups of $S$ are trivial.} inverse monoid; see \cite[Example 1.21(d), item (iv)]{Leech95}. Of course, the requirement of $S$ being a monoid is not essential: if a semigroup $S$ periodic and combinatorial then so is the monoid $S^1$ obtained by adjoining a formal identity to $S$. Clearly, if a semigroup satisfies \eqref{eq:comb}, then it is both periodic and combinatorial whence Leech's observation applies. We have preferred the above direct proof of Lemma~\ref{lem:b21-as-ais} because we need a $(\cdot,{}^{-1})$-term for the semilattice operation, and such a term is not explicitly present in~\cite{Leech95}.
\end{remark}

Obviously, the 6-element Brandt monoid satisfies the identity $x^2\approx x^3$. Thus, Lemma~\ref{lem:b21-as-ais} applies, and $(B_2^1,\oplus,\cdot)$ is an \ais.
It is known (and easy to verify) that $\oplus$ is the only addition on $B_2^1$ under which $B_2^1$ becomes an \ais.

Our main result states that, similarly to the plain semigroup $(B_2^1,\cdot)$ and the inverse semigroup $(B_2^1,\cdot,{}^{-1})$, the \ais\ $(B_2^1,\oplus,\cdot)$ admits no finite identity basis. Its proof employs a series of inverse semigroups $C_n$, $n=2,3,\dotsc$, constructed in~\cite{Kle79} as semigroups of partial one-to-one transformations. Here, to align with the matrix representation chosen for the $B_2^1$, we describe them as semigroups of zero-one matrices.

The set $R_m$ of all zero-one $m\times m$-matrices which have at most one entry equal to 1 in each row and column forms an inverse monoid under usual matrix multiplication $\cdot$ and transposition. The inverse monoid $R_m$ is called the \emph{rook monoid}\footnote{The rook monoid is nothing but the matrix representation of the \emph{symmetric inverse monoid}; see \cite[Section IV.1]{Pet84} or \cite[p. 6]{Law99}. The name `rook monoid' was suggested by Solomon \cite{Sol02}.} as its matrices encode placements of nonattacking rooks on an $m\times m$ chessboard.

Let $m=2n+1$ and define $m\times m$-matrices $c_1,\dots,c_n$ by
\[
c_k:=E_{k+1\,k}+E_{n+k\ n+k+1},\ \ k=1,\dots,n,
\]
where, as usual, $E_{ij}$ denotes the $m\times m$-matrix unit with an entry 1 in the $(i,j)$ position and 0's elsewhere. For instance, if $n=2$, then $c_1$ and $c_2$ are the following $5\times 5$-matrices:
\[
c_1=\begin{pmatrix}
0 & 0 & 0 & 0 & 0\\
1 & 0 & 0 & 0 & 0\\
0 & 0 & 0 & 1 & 0\\
0 & 0 & 0 & 0 & 0\\
0 & 0 & 0 & 0 & 0
\end{pmatrix},\qquad
c_2=\begin{pmatrix}
0 & 0 & 0 & 0 & 0\\
0 & 0 & 0 & 0 & 0\\
0 & 1 & 0 & 0 & 0\\
0 & 0 & 0 & 0 & 1\\
0 & 0 & 0 & 0 & 0
\end{pmatrix}.
\]
Let $C_n$ be the inverse subsemigroup of the rook monoid $R_m$ generated by the matrices $c_1,\dots,c_n$. As a plain subsemigroup, $C_n$ is generated by $c_1,\dots,c_n$ and their inverses (i.e., transposes) $c_1^{-1},\dots,c_n^{-1}$.

The next lemma collects properties of the semigroups $C_n$ that we need.
\begin{lemma}
\label{lem:cn}
\emph{(i)} The semigroup $(C_n,\cdot)$ does not belong to the semigroup variety generated by the monoid $(B_2^1,\cdot)$.

\emph{(ii)} The semigroup $(C_n,\cdot)$ satisfies the identity $x^2\approx x^3$.

\emph{(iii)} For each $k=1,\dots n$, $M_k(n):=C_n\setminus\{c_k,c_k^{-1}\}$ forms an inverse subsemigroup of the inverse semigroup $(C_n,\cdot,{}^{-1})$.

\emph{(iv)} For each $k=1,\dots n$, the inverse semigroup $(M_k(n),\cdot,{}^{-1})$  belongs to the inverse semigroup variety generated by the inverse monoid $(B_2^1,\cdot,{}^{-1})$.
\end{lemma}

\begin{proof}
(i) This property was established in \cite[Lemma 3]{Kle82} by exhibiting, for each $n\ge 2$, a semigroup identity that holds in $(B_2^1,\cdot)$ and fails in $(C_n,\cdot)$.

(ii) This is easy to verify (and also follows from the proof of Lemma 1 in \cite{Kle79}).

(iii) This is clear (and is a part of Lemma 1 in \cite{Kle79}).

(iv) This is Property (C) in \cite{Kle79}.
\end{proof}

\begin{remark}
Items (i)--(iii) of Lemma \ref{lem:cn} are easy. In contrast, the proof of (iv) in \cite{Kle79} is long and complicated. We mention in passing that now the proof can be radically simplified by using a deep result by Ka\softd{}ourek \cite{Kad91} who provided an effective membership test for the inverse semigroup variety generated by $(B_2^1,\cdot,{}^{-1})$.
\end{remark}

\begin{theorem}
\label{thm:b21-ais-nfb}
The semiring identities of the additively idempotent semiring $(B_2^1,\oplus,\cdot)$ admit no basis involving only finitely many variables, and hence, no finite basis.
\end{theorem}

\begin{proof}
Arguing by contradiction, assume that $(B_2^1,\oplus,\cdot)$ has an identity basis $\Sigma$ such that each identity $u\approx v$ in $\Sigma$ involves less than $n$ variables. Consider the inverse semigroup $(C_n,\cdot,{}^{-1})$. By Lemmas~\ref{lem:b21-as-ais} and~\ref{lem:cn}(ii), the addition defined by $x\oplus y:=(xy^{-1})^2x$ makes $(C_n,\oplus,\cdot)$ an \ais. Consider an arbitrary evaluation $\varepsilon$ of variables $x_1,\dots,x_\ell$ involved in the identity $u\approx v$ in this \ais. By the pigeonhole principle, there exists an index $k\in\{1,\dots n\}$ such that neither $c_k$ nor $c_k^{-1}$ belongs to the set $\{\varepsilon(x_1),\dots,\varepsilon(x_\ell)\}$ as this set contains at most $\ell<n$ elements. Thus, $\{\varepsilon(x_1),\dots,\varepsilon(x_\ell)\}\subset M_k(n)$.

Since $x\oplus y$ expresses as $(\cdot,{}^{-1})$-term, one can rewrite the identity $u\approx v$ into an identity $u'\approx v'$ in which $u'$ and $v'$ are $(\cdot,{}^{-1})$-terms.
Since $u\approx v$ holds in $(B_2^1,\oplus,\cdot)$, the rewritten identity $u'\approx v'$ holds in the inverse semigroup $(B_2^1,\cdot,{}^{-1})$. By Lemma \ref{lem:cn}(iv) the latter identity holds also in the inverse semigroup $(M_k(n),\cdot,{}^{-1})$, and so $u'$ and $v'$ take the same value under every evaluation of the variables $x_1,\dots,x_\ell$
in $M_k(n)$. Hence $\varepsilon(u)=\varepsilon(u')=\varepsilon(v')=\varepsilon(v)$. We conclude that the identity $u\approx v$ holds in the \ais\ $(C_n,\oplus,\cdot)$. Since an arbitrary identity from $\Sigma$ holds in $(C_n,\oplus,\cdot)$, this \ais\ belongs to the \ais\ variety generated by $(B_2^1,\oplus,\cdot)$. This, however, contradicts Lemma \ref{lem:cn}(i), according to which even the semigroup reduct $(C_n,\cdot)$, does not belongs to semigroup variety generated by $(B_2^1,\cdot)$.
\end{proof}

\begin{remark}
To the best of my knowledge, the result of Theorem~\ref{thm:b21-ais-nfb} has not been published up to now. However, after preparing the present article I have learnt that the result has also been obtained by colleagues in Xi'an and Melbourne but with an entirely unrelated proof.
\end{remark}

I mention also a related paper by Dolinka \cite{Dol07} where he introduces a 7-element ai-semiring denoted $\Sigma_7$ and proves that its identities are not finitely based. The semigroup reduct of $\Sigma_7$ is just the monoid $B_2^1$ with an extra zero adjoined so that $(\Sigma_7,\cdot,{}^{-1})$ and $(B_2^1,\cdot,{}^{-1})$ satisfy the same inverse semigroup identities. However, the addition in $\Sigma_7$ is not derived from its inverse semigroup structure, and one can easily see that the semiring identities of $(\Sigma_7,+,\cdot)$ and $(B_2^1,\oplus,\cdot)$ are essentially different. It should be also mentioned that in \cite{Dol07} Dolinka actually considers \ais{}s with 0 as algebras of type (2,2,0).

\begin{remark}
Leech~\cite{Leech95} defined an \emph{inverse algebra} as an algebra $(A,\cdot,\wedge,{}^{-1},1)$ of type $(2,2,1,0)$ such that the reduct  $(A,\cdot,{}^{-1},1)$  is an inverse monoid, the reduct $(A,\wedge)$ is a meet semilattice, and the natural partial order of the inverse monoid coincides with that of the semilattice. Clearly, $(B_2^1,\cdot,\oplus,{}^{-1},E)$ constitutes an inverse algebra in Leech's sense, and the above proof of Theorem~\ref{thm:b21-ais-nfb} can be easily adapted to show that $(B_2^1,\cdot,\oplus,{}^{-1},E)$ has no finite identity basis also as such algebra.
\end{remark}


\small

\begin{thebibliography}{99}
\bibitem{BuSa81}
S. Burris and H.P. Sankappanavar, A Course in Universal Algebra. Springer-Verlag, Berlin-Heidelberg-New York (1981)

\bibitem{Dol07}
I. Dolinka, A nonfinitely based finite semiring. Int. J. Algebra Comput. \textbf{17}(8), 1537--1551 (2007)

\bibitem{JackZhang21}
M. Jackson and W.T. Zhang, From $A$ to $B$ to $Z$. Semigroup Forum (in print)

\bibitem{Kad91}
J. Ka\softd{}ourek, On varieties of combinatorial inverse semigroups. I. Semigroup Forum \textbf{43}, 305--330 (1991)

\bibitem{Kle77}
E.I. Kleiman, On bases of identities of Brandt semigroups. Semigroup Forum \textbf{13}, 209--218 (1977)

\bibitem{Kle79}
E.I. Kleiman, Bases of identities of varieties of inverse semigroups. Sib. Math. J. \textbf{20}, 530--543 (1979). [Translated from Sibirskii Matematicheskii Zhurnal \textbf{20}, 760--777 (1979)]

\bibitem{Kle82}
E.I. Kleiman, A pseudovariety generated by a finite semigroup. Ural. Gos. Univ. Mat. Zap. \textbf{13}(1), 40--42 (1982) (Russian)

\bibitem{KP05}
M. Ku\v{r}il and L. Pol\'{a}k, On varieties of semilattice-ordered semigroups. Semigroup Forum, {\bf 71}, 27--48  (2005)

\bibitem{Law99}
M.V. Lawson, Inverse Semigroups. The Theory of Partial Symmetries. World Scientific, Singapore (1999)

\bibitem{LeeZhang2015}
E.W.H. Lee and W.T. Zhang, Finite basis problem for semigroups of order six. London Math. Soc. J. Comput. Math. \textbf{18}, 1--129 (2015)

\bibitem{Leech95}
J. Leech,  Inverse monoids with a natural semilattice ordering. Proc. London Math. Soc. \textbf{s3-70}(1), 146-182 (1995)

\bibitem{McA07}
D.B. McAlister, Semilattice ordered inverse semigroups.  In J.M. Andr\'e et al (eds.), Semigroups and Formal Languages, pp. 205--218. World Scientific, New Jersey (2007)

\bibitem{Per66}
P. Perkins, Decision Problems for Equational Theories of Semigroups and General Algebras. Ph.D. thesis, Univ.\ of California, Berkeley (1966)

\bibitem{Per69}
P. Perkins, Bases for equational theories of semigroups, J. Algebra \textbf{11}, 298--314 (1969)

\bibitem{Pet84}
M. Petrich, Inverse Semigroups. John Wiley \& Sons, New York (1984)

\bibitem{Sch73}
B.M. Schein, Completions, translational hulls and ideal extensions of inverse semigroups. Czechoslovak Math. J. \textbf{23}(4), 575--610 (1973)

\bibitem{Sol02}
L. Solomon, Representations of the rook monoid. J. Algebra \textbf{256}(2), 309--342 (2002)
\end{thebibliography}
\end{document}